\documentclass[reqno]{amsart}
\usepackage{amsthm,amsmath,amssymb,amscd, mathrsfs,graphics, hyperref,enumitem}
\usepackage{diagrams}
\usepackage[all,cmtip]{xy} 
\usepackage{verbatim} 

\newarrow {Mapsto} |--->
\newarrow {Iso} ---->

\newcommand{\B}{\cb}

\newcommand{\Ch}{\mathrm{Ch}}

\newcommand{\IGX}{I_GX}

\newcommand{\inCh}{\ovCh}

\newcommand{\IzeroX}[1]{I_G^{0,V}X}

\newcommand{\Lr}[1]{LR(#1)}

\newcommand{\Mbar}{\overline{\cM}}
\newcommand{\N}{\mathbf{N}}              
\newcommand{\odeg}{\deg_{\mathrm orb}}

\newcommand{\R}{\mathscr{R}}             
\newcommand{\Sm}{\mathscr{S}}             
\newcommand{\Smv}{^V\kern-.5em\mathscr{S}}

\newcommand{\T}{\mathbb{T}}    
\newcommand{\TV}{^V\kern-.4em\mathscr{T}}

\renewcommand{\ar}{R} 
\newcommand{\as}{S} 
\newcommand{\age}{\operatorname{age}}

\newcommand{\bra}[1]{ {\{ {#1} \}} }

\newcommand{\cM }{\mathscr{M}}

\newcommand{\cb}{{\mathscr B}}

\newcommand{\cf}{{\mathscr F}}

\newcommand{\co}{{\mathscr O}}

\newcommand{\crr}{{\mathscr R}}

\newcommand{\cs}{{\mathscr S}}

\newcommand{\euler}{\operatorname{eu}}

\newcommand{\g}{\mathfrak{g}}

\newcommand{\ix}{\mathscr{X}}

\newcommand{\m}{\mathfrak{m}}

\newcommand{\nc}{\mathbb{C}}

\newcommand{\nq}{\mathbb{Q}}

\newcommand{\nz}{\mathbb{Z}}
 
\newcommand{\oCh}{\mathscr{C}\kern-.25em{h}} 
\newcommand{\ovCh}{\widetilde{\oCh}}

\newcommand{\ok}{K_{orb}}

\newcommand{\one}{\mathbf{1}}

\newcommand{\res}[2]{\left.{#1}\right|_{#2}}

\newcommand{\rk}{\mathrm{rk}}

\newcommand{\spec}{\operatorname{Spec}}

\newcommand{\stack}[1]{\mathscr {#1}}

\newcommand{\ti}{\star}

\newcommand{\A}{\mathbb A}

\newcommand{\Pro}{\mathbb P}
\newcommand{\Q}{\mathbb Q}

\newcommand{\Z}{\mathbb Z}

\newcommand{\Iner}{\mathfrak{I}} 

\newcommand{\IIX}{{\Iner \kern-.75em \Iner}_{\ix}}
\newcommand{\II}{\Itwo}

\DeclareMathOperator{\Spec}{Spec}

\DeclareMathOperator{\ch}{ch}
\DeclareMathOperator{\Td}{Td}

\newcommand{\Imult}[1]{\mathbb{I}^{#1}}  
\newcommand{\Itwo}{\Imult{2}}

\newtheorem{thm}{Theorem}[subsection] 
\newtheorem{lm}[thm]{Lemma}
\newtheorem{prop}[thm]{Proposition} 
\newtheorem{crl}[thm]{Corollary}

\theoremstyle{definition}
\newtheorem{rem}[thm]{Remark} 

\newtheorem{df}[thm]{Definition} 

\newtheorem{df-pr}[thm]{Definition-Proposition}

\theoremstyle{remark} 
 \renewcommand{\thenota}{\kern-1ex}

\begin{document}
\title[Inertial Products] {A plethora of inertial products}

\subjclass[2010]{Primary: 55N32, 55N15.  Secondary:
14N35, 
53D45,
57R18,
19L10, 
19L47%
}

\author [D. Edidin]{Dan Edidin} \address {Department of
Mathematics, University of Missouri, Columbia, MO 65211, USA}
\email{edidind@missouri.edu} 

\author [T. J. Jarvis]{Tyler J. Jarvis} \address {Department of
Mathematics, Brigham Young University, Provo, UT 84602, USA}
\email{jarvis@math.byu.edu} 

\author [T. Kimura]{Takashi Kimura} \address {Department of
Mathematics and Statistics; 111 Cummington Street, Boston University;
Boston, MA 02215, USA } \email{kimura@math.bu.edu} 
\thanks{Research of the second author partially supported by NSA grant H98230-10-1-0181.
Research of the third author partially supported by NSA grant H98230-10-1-0179.}  

\date{\today}

\begin{abstract}  
For a smooth Deligne-Mumford stack $\ix$, we describe a large number of {\em inertial products} on $K(I\ix)$ and $A^*(I\ix)$ and \emph{inertial Chern characters}.  We do this  by developing a theory of {\em inertial pairs}.   Each inertial pair determines an inertial product on $K(I\ix)$ and an inertial product on $A^*(I\ix)$ and Chern character ring homomorphisms between them. 
  We show that there are many inertial pairs; indeed, 
  every vector bundle $V$ on $\ix$
  defines two new inertial pairs. We recover, as special
  cases, both the orbifold products of \cite{ChRu:04, AGV:02, FaGo:03, JKK:07, 
    EJK:10} and the virtual product of \cite{GLSUX:07}. 

  We also introduce an entirely new product we call the
  \emph{ localized orbifold product}, which is defined on
  $K(I\ix)\otimes\nc$.

The inertial products developed in this paper are used in a subsequent paper \cite{EJK:12b} to describe a theory of inertial Chern classes and power operations in inertial K-theory.  These constructions provide new manifestations of mirror
symmetry, in the spirit of the Hyper-K\"ahler Resolution Conjecture.
\end{abstract}

\maketitle 
\section{Introduction}
The purpose of this note is to describe a large number of \emph{inertial products} and Chern characters by developing a formalism of \emph{inertial
pairs}. An inertial pair for a Deligne-Mumford stack $\ix$ is a pair
$(\crr, \cs)$, where $\crr$ is a vector bundle on the
double inertia stack $\Itwo\ix$ and $\cs$ is a non-negative,
rational $K$-theory class on the inertia stack $I\ix$ satisfying
certain compatibility conditions. For stacks with finite
stabilizer, an inertial pair determines inertial products on
cohomology, Chow groups, and $K$-theory of $I\ix$. 
In Chow and cohomology, this product 
respects an orbifold grading equal to the ordinary grading corrected by 
the virtual rank of $\cs$ (or age).
An inertial pair also allows us to define an
inertial Chern character, which is a ring homomorphism for the new inertial products.

The motivating example of an inertial pair is the orbifold pair
$(\crr, \cs)$, where $\crr$ is the obstruction
bundle coming from orbifold Gromov-Witten theory, and $\cs$ is the
class defined in \cite{JKK:07}. The corresponding product is the
Chen-Ruan orbifold product, and the Chern character is the one defined
in \cite{JKK:07}. One of the results of this paper is that every
vector bundle $V$ on a Deligne-Mumford stack determines two inertial
pairs $(\crr^+V, \cs^+V)$ and $(\crr^{-}V,
\cs^{-}V)$. The $+$ product corresponds to the orbifold
product on the total space 
of the bundle $V$, but the
$-$ product is twisted by an
isomorphism
and does not directly
correspond to an orbifold product on a bundle. However, we prove (Theorem
\ref{thm.star-geo})
that there is an automorphism of the total Chow group
$A^*(I\ix) \otimes {\nc}$ 
 which induces a ring isomorphism between the 
 $-$ product for $V$ and 
the $+$ product for $V^*$.  A similar result also holds for cohomology.

When $V = {\mathbb T}$ is the tangent bundle
of ${\mathscr X}$, we show that the virtual product considered by
\cite{GLSUX:07} is the product associated to the inertial pair
$(\crr^{-}{\mathbb T}, \cs^{-}{\mathbb T})$. 
It follows, 
after tensoring with ${\mathbb C}$,
that the virtual 
orbifold
Chow ring is isomorphic (but not equal) to the 
Chen-Ruan orbifold Chow ring of the cotangent bundle ${\mathbb T}^*$.
Our result also implies that there is a corresponding Chern character ring homomorphism for the virtual product.

In the final section we show that in certain cases, even if $V$
is not a vector bundle but just an element of K-theory, 
we can still determine a product in localized $K$-theory. This allows us to define
a new product on 
$K(I\ix)\otimes {\nc}$, 
which we call the \emph{ localized product.}

In a subsequent paper \cite{EJK:12b} we will show that for {\em Gorenstein 
inertial pairs} (such as the one determining the virtual product) there is a theory of Chern classes and compatible power operations on inertial $K$-theory. This will be used to give further manifestations of mirror symmetry on hyper-K\"ahler Deligne-Mumford stacks.

\subsection*{Review of Previous Related Work}
Because there has been much work in this area by many authors from different areas of mathematics, we give a brief overview here of previous work to help put the current paper in context.

In 2000, inspired by physicists \cite{DHVW:85, DHVW:86}, Chen-Ruan  \cite{ChRu:02} developed a new product on the the cohomology of the inertia $I\ix$ of an almost complex orbifold $\ix$.  
In 2001, Fantechi-G\"ottsche \cite{FaGo:03} showed that  
when
the orbifold $\ix$ was a global quotient $[X/G]$ by a finite group, 
the Chen-Ruan orbifold cohomology ring $H_{CR}(\ix)$ was the $G$-invariant subring of $H_{FG}(X,G)$, the cohomology of the inertia manifold $IX$ endowed with a certain noncommutative product.
It followed that if $X$ is the symmetric product of a surface with trivial canonical class, then 
the orbifold cohomology of $\ix$ is 
isomorphic to the cohomology ring of the Hilbert scheme, 
as predicted by
the Hyper-K\"ahler Resolution Conjecture \cite{Rua:06}.

About the same time Kaufmann presented an axiomatic approach to orbifolding Frobenius algebras \cite{Kau:02,Kau:03}, and described how the Fantechi-G\"ottsche construction  fit into this framework \cite{Kau:pers}.

Adem-Ruan \cite{AdRu:03} then studied the K-theory 
$K(\ix)$ 
of a global quotient orbifold $\ix = [X/G]$, where $G$ is a 
 Lie
group and they also studied the twisted K-theory of $[X/G]$.  They did not construct a new ``orbifold'' product on 
$K(\ix)$, but they did show that there is a Chern character that gives a vector space isomorphism from  $K(\ix)$ 
to $H_{CR}(\ix)$.  This Chern character is \emph{not} a ring homomorphism.  Tu-Xu \cite{TuXu:06} later extended this result to more general twistings and orbifolds.  

Abramovich-Graber-Vistoli \cite{AGV:02} constructed an algebraic version $A_{AGV}(\ix)$ of the Chen-Ruan cohomology, producing the corresponding product on the Chow group $A_{AGV}(\ix) =A^*(I\ix)$ of the inertial stack $I\ix$ of a (smooth) Deligne-Mumford stack with projective coarse moduli space.   

In all of these constructions the basic idea is to use an analogue of the moduli space $\Mbar_{0,3}(\ix, 0)$ of genus-zero, three-pointed, orbifold (or $G$-equivariant) stable maps into $\ix$.  This space has three \emph{evaluation maps} $e_i:\Mbar_{0,3}(\ix, 0) \rTo I\ix$, and the structure constants $\langle \alpha_1, \alpha_2, \alpha_3 \rangle$ for the new product on $I\ix$ are given by computing $\int_{\Mbar_{0,3}(\ix,0)} \prod_{i=1}^3 e_i^*(\alpha_i) \cdot \euler(\R)$, where $\euler(\R)$ is the top Chern class of an obstruction bundle on $\Mbar_{0,3}(\ix, 0)$.  
The
main difficulty in computing the new product was computing the obstruction bundle $\R$ and its top Chern class.  

In 2004, Chen-Hu \cite{ChHu:06} produced a formula for the obstruction class in the  case of \emph{Abelian} orbifolds and used it to describe a deRham model for the Chen-Ruan product. In 2005, Jarvis-Kaufmann-Kimura \cite{JKK:07} proved a simple, intrinsic formula for the obstruction bundle $\R$ for general (not just Abelian) orbifolds, requiring no mention of stable curves or moduli spaces of maps.
In the Abelian case this formula reduces to Chen-Hu's result.  In \cite{JKK:07}, that formula is 
used
to do several things:
\begin{enumerate}
\item Create  Chow- (respectively, K-) theoretic analogues of the Fantechi-G\"oettsche ring $H_{FG}(X,G)$ whose rings of invariants is
the AGV ring $A_{AGV}(\ix)$ (respectively, a ring whose underlying vector space
is $K(\ix)$ of Adem-Ruan). 
Corresponding products twisted by discrete torsion were also introduced.
\item Define a new 
(orbifold) product on the  K-theory $\ok(\ix)$ of the inertia $I\ix$, for any smooth Deligne-Mumford stack $\ix$.
\item Define an orbifold Chern character ring homomorphism from the new orbifold K-theory rings to the corresponding Chow or cohomology rings
This new Chern character is
a deformation of the ordinary Chern character, as the latter fails to preserve the orbifold multiplications. 
\item Outline how the same formula and formalism 
may be used to give 
analogous results in other categories, 
e.g.
equivariant structures on almost complex manifolds with a Lie group action.
\end{enumerate}
About the same time, Adem-Ruan-Zhang \cite{ARZ:08} independently defined an orbifold product on twisted $\ok(\ix)$, 
and in the case of a global quotient by a finite group, Kaufmann-Pham \cite{KaPh:09} connected this to the twisted Drinfel'd double of the group ring.  

Becerra-Uribe in \cite{BeUr:07} extended these results 
to the equivariant setting for global quotients by infinite Abelian groups, and in \cite{EJK:10}, we extended these results to an equivariant setting for 
global quotients by general (non-Abelian) infinite groups
by introducing a variant of the formula for the obstruction bundle in \cite{JKK:07}.

A recent paper \cite{HuWa:11} repeats the description of  the orbifold product of \cite{JKK:07,ARZ:08}, the formula  of \cite{JKK:07,EJK:10} for the obstruction class, and the  Chern character ring homomorphism of \cite{JKK:07,EJK:10} in the almost-complex setting, as originally described in Section 10 of \cite{JKK:07}.

In \cite{BGNX:07, BGNX:11, GLSUX:07, LUX:08}, a different product in inertial Chow
and inertial cohomology
 theory, analogous to the Chas-Sullivan product \cite{ChSu:99} on loop spaces, was
 introduced. This so-called 
 \emph{virtual (orbifold) product} 
 is a special case of the constructions of this paper, 
 cf.~Section~\ref{sec:virprod}. Surprisingly, it is \emph{not} equal to the orbifold product for
 the cotangent bundle. 
 We note, however, that after tensoring with $\nc$, both the orbifold Chow and orbifold cohomology (but not orbifold K-theory) of the cotangent bundle are isomorphic to their virtual counterparts.

Kaufmann was the first to study the possibility of many stringy products on Frobenius algebras in settings involving functors other than just K-theory, cohomology, and Chow theory (see \cite{Kau:02,Kau:03,Kau:10a}).
He treated this primarily as an algebraic question and reformulated the problem of constructing a stringy product in terms of certain cocycles.  It is not \emph{a priori} clear that there should always exist a stringy product, but \cite{Kau:10a} shows how to extend the ideas of \cite{JKK:07} to prove existence of at least one stringy product for his more general setting.  In some cases he can also show uniqueness of the product \cite{Kau:04a}.

Finally, we note that Pflaum-Postuma-Tang-Tseng \cite{PPTT:07} have shown that the Hochschild cohomology of a certain algebra attached to a groupoid presentation of a
symplectic orbifold is isomorphic to the cohomology of the inertia orbifold as vector spaces.  The product in Hochschild cohomology induces a product on the cohomology of the inertia orbifold.  It would be interesting to understand the relation of that product to the the products described in this paper.

\subsection*{Acknowledgements} 
We wish to thank 
the \emph{Algebraic Stacks: Progress and Prospects} Workshop at BIRS for 
their support where part of this work was done.
The second author also wishes to thank Dale Husem\"oller for helpful
conversations and both the Max Planck Institut f\"ur Mathematik in
Bonn and the Institut Henri Poincar\'e for their generous support of
this research.  The third author wishes to thank Yunfeng Jiang and
Jonathan Wise for helpful conversations and the Institut Henri
Poincar\'e, where part of this work was done, for their generous
support.

\section{Background material from \cite{EJK:10}}
To make this paper self-contained, we recall some background material from the paper \cite{EJK:10}. 

\subsection{Background notation}

We work entirely in the complex algebraic category. We will work
exclusively 
with smooth Deligne-Mumford stacks $\ix$ which have \emph{finite
stabilizer}, by which we mean the inertia
map $I\ix \rTo \ix$ is finite. We will also assume that every stack
$\ix$ has the {\em resolution property}. This means that every
coherent sheaf is the quotient of a locally free sheaf. This
assumption has two consequences. The first is that the natural map
$K(\ix) \rTo G(\ix)$ is an isomorphism, where $K(\ix)$ is the
Grothendieck ring of vector bundles and $G(\ix)$ is the Grothendieck
group of coherent sheaves. The second consequence  is
that $\ix$ is a {\em quotient stack} \cite{EHKV:01}. This means that $\ix = [X/G]$,
where $G$ is a linear algebraic group acting on a scheme or algebraic
space $X$.

If $\ix$ is a smooth Deligne-Mumford stack, then we will implicitly
choose a presentation $\ix = [X/G]$. This allows us to identify the
Grothendieck ring $K(\ix)$ with the equivariant Grothendieck ring
$K_G(X)$, and the Chow ring $A^*(\ix)$ with the equivariant
Chow ring $A^*_G(X)$. We will use the notation $K(\ix)$ and 
$K_G(X)$ (resp. $A^*(\ix)$ and  $A^*_G(X))$  interchangeably.

\begin{df}
  Let $G$ be an algebraic group acting on a scheme or algebraic space
  $X$. We define the \emph{inertia space}  $$\IGX:= \{(g,x)\,|\, gx
  = x\} \subset G\times X.$$  There is an induced action of $G$ on
  $\IGX $ given by $g \cdot (m,x) = (gmg^{-1}, gx)$. The quotient
  stack $I\ix := [\IGX /G] $ is the 
  \emph{inertia stack}
of the quotient stack ${\stack X} := [X/G]$.

  More generally, we define the \emph{higher inertia spaces} to be the
  $k$-fold fiber products $$\Imult{k}_G X = \IGX  \times_X \ldots \times_X
  \IGX
  = \{(m_1,\ldots,m_k,x)\,|\,m_i x = x,\, \forall i=1,\ldots,k\} \subset G^k\times X. $$
The quotient stack $\Imult{k}\ix := \left[\Imult{k}_G X /G\right]$ 
is the corresponding \emph{higher inertia stack.}
\end{df}
The assumption that $\ix$ has finite stabilizer means that
the projection
$\IGX \rTo X$ is a finite morphism. The composition $\mu \colon G \times G \rTo G$ induces a composition
$\mu \colon \IGX \times_X \IGX \rTo \IGX$.
This composition makes $\IGX $ into an $X$-group
with identity section $X \rTo \IGX $ given by $x \mapsto (1,x)$. 

\begin{df}
Let $G^\ell$ be a $G$-space with the diagonal conjugation action. 
A \emph{diagonal conjugacy class} is a $G$-orbit $\Phi\subset G^\ell$.  
\end{df}

\begin{df}
For all $m$ in $G$, let $X^m = \{(m,x) \in \IGX \}$. For all $(m_1,\ldots,m_\ell)$ in $G^\ell$, let $X^{m_1,\ldots,m_\ell} = \{\,(m_1,\ldots,m_\ell,x)\in \Imult{\ell}_G X \}.$
For all conjugacy classes $\Psi \subset G$, let $I(\Psi) = \{(m,x)\in\IGX\,|\,m\in \Psi \}$. 
More generally, for all diagonal conjugacy classes $\Phi\subset G^\ell$, 
let $\Imult{\ell}(\Phi) = \{(m_1, \ldots m_\ell, x)\in \Imult{\ell}_G X\, |\, (m_1, \ldots , m_\ell) \in \Phi\}$. 
\end{df}

By definition, $I(\Psi)$ and $\Imult{\ell}(\Phi)$ are $G$-invariant subsets of $\IGX $
and $\Imult{\ell}_G X$, respectively. 
If $G$ acts 
with finite stabilizer  on $X$, then $I(\Psi)$ is empty unless $\Psi$
consists of elements of finite order. Likewise, $\Imult{\ell}(\Phi)$ is empty unless every $\ell$-tuple $(m_1, \ldots , m_\ell) \in \Phi$ generates a finite group. Since conjugacy classes of elements of finite order are closed, $I(\Psi)$ and $\Imult{\ell}(\Phi)$
are closed.
\begin{prop} \label{prop.inertiadecomp}
(\cite[Prop.~2.11, Prop.~2.17]{EJK:10}) If $G$
  acts 
properly on $X$, then $I(\Psi) = \emptyset$ for all but finitely many
  conjugacy classes
   $\Psi$ and the $I(\Psi)$ are unions of connected components
  of $\IGX $. Likewise, $\Imult{l}(\Phi)$ is empty for all but finitely
  many 
  diagonal conjugacy classes $\Phi\subset G^\ell$
  and each $\Imult{l}(\Phi)$ is a union of
  connected components of $\Imult{l}_G X$.
\end{prop}

We frequently work with a group $G$ acting on a space $X$ where
the quotient stack $[X/G]$ is not connected. As a consequence,
some care is required in the definition of the rank and Euler class of a vector
bundle.  Note that for any $X$, the group $A^0_G(X)$ satisfies $A^0_G(X) =\Z^{\ell}$, 
where $\ell$
is the number of connected components of the quotient stack $\ix = [X/G]$.
\begin{df}
If $E$ is an equivariant vector bundle on $X$, then we define the
{\em rank} of $E$ to be $\rk(E) := \Ch^0({ E}) \in \Z^{\ell} =
A^0_G(X)$. Note that the rank of $E$ lies in the semi-group ${\mathbb N}^{\ell}$, 
where ${\mathbb N} = \{0,1,2,\dots\}$.
If $E_1, \ldots, E_n$
are vector bundles, then the {\em virtual rank (or augmentation)}
of the element
$\sum_{i=1}^n n_i [E_i] \in K_G(X)$ 
is the weighted sum 
$\sum_i n_i \rk({\mathscr E}_i) \in \Z^{\ell}$.
\end{df}

If $E$ is a $G$ equivariant vector bundle on $X$,
then the  rank of $E$ on the connected components of $\ix = [X/G]$ 
is bounded (since we assume that $\ix$ has finite type).
\begin{df}\label{df:euler}
If $E$ is a $G$-equivariant vector bundle on $X$, we call 
the element
$\lambda_{-1}(E^*) = \sum_{i=0}^{\infty} (-1)^i  [\Lambda^i E^*] \in K_G(X)$ 
the \emph{K-theoretic Euler
    class} of $E$. 
(Note that this sum is finite.)

Likewise, we define the element $c_{\textbf{top}}( E) \in A^*_G(X)$,
corresponding to the sum of the top Chern classes of $E$ on
each connected component of $[X/G]$, to be the \emph{Chow-theoretic
  Euler class} of $E$.  These definitions can be extended to
any non-negative element by multiplicativity.  It will be convenient
to use the symbol $\euler(\cf)$ to denote both of these Euler classes for a
non-negative element $\cf \in K_G(X)$.
\end{df}
\subsection{The logarithmic restriction and twisted pullback}\label{sec:LR}
We recall a construction from \cite{EJK:10} that will be used several
times throughout the paper.  However, to improve clarity, we use slightly different notation than in \cite{EJK:10}. 

\begin{df} \label{def.logtrace}\cite{EJK:10}
Let $X$ be an
algebraic space with an action of an algebraic group $Z$.  Let $E$ be a rank-$n$ vector bundle on $X$ and let $g$ be  
a unitary automorphism of the fibers of $E \rTo X$. If we assume  that the action of $g$ 
commutes with the action of $Z$ on $E$, 
the eigenbundles for the action of $g$ are all $Z$-subbundles.
Let $\exp(2\pi \sqrt{-1} \alpha_1), \ldots ,\exp(2\pi \sqrt{-1}
\alpha_r)$ be the distinct eigenvalues of $g$ acting on $E$, 
with $0\le \alpha_k <1$ for all $k\in\{1\dots,r\}$,
and let $E_1, \ldots , E_r$ be the corresponding eigenbundles.  

We define the \emph{logarithmic trace  of $E$} by the formula 
\begin{equation} \label{eq.logtracebund}
L(g)(E) = \sum_{k =1}^{r} \alpha_k E_k \in
K_Z(X) \otimes {\mathbb R}
\end{equation}
on each connected component of $X$.
\end{df}

The following key fact about the logarithmic trace was proved in \cite{EJK:10}.
\begin{prop} \label{prop.logrestriction}\cite[Prop.  4.6]{EJK:10}
Let ${\mathbf g} = (g_1, \ldots , g_{\ell})$ be an $\ell$-tuple of elements of a compact subgroup of a reductive group $H$, satisfying $\prod_{i=1}^{\ell} g_i = 1$. 
Let $X$ be an
algebraic space with an action of an algebraic group $Z$, and let $V$ be a $(Z 
\times H)$-equivariant bundle on $X$, where $H$ is assumed to act trivially on $X$. The element 
\begin{equation*}
\sum_{i=1}^{\ell} L(g_i)(V) -V + V^{{\mathbf g}}
\end{equation*}
in $K_Z(X)$ is represented by a $Z$-equivariant vector bundle.
\end{prop}
Using Proposition \ref{prop.logrestriction} we make the following definition.
\begin{df}Let $G$ be an algebraic group acting quasi-freely on an algebraic space,
and let $V$ be a $G$-equivariant vector bundle on $X$.
Given ${\mathbf g} = (g_1, \ldots g_{\ell}) \in G^{\ell}$,
if the $g_i$ all lie in a common compact subgroup
and satisfy $\prod_{i=1}^{\ell} g_i = 1$, then set
\begin{equation*}
V({\mathbf g}) = \sum_{i=1}^{\ell} L(g_i)(V|_{X^{\mathbf g}}) - 
V|_{X^{\mathbf g}} + V^{{\mathbf g}}|_{X^{{\mathbf g}}}.
\end{equation*}
\end{df}

We wish to extend this definition to give a map from $K_G(X)$  to $K_G(\IGX)$, but we must first understand the decomposition of $K_G(\IGX)$ and $A^*_G(\IGX)$ by conjugacy classes.

As a consequence of Proposition \ref{prop.inertiadecomp}, we see that $K_G(I_GX)$
(resp.~$A^*_G(\IGX )$) is a direct sum of the of $K_G(I(\Psi))$
(resp.~$A^*_G(I(\Psi))$) as $\Psi$ runs over conjugacy classes of
elements of finite order in $G$.  A similar statement holds for the
equivariant K-theory and Chow groups of the higher inertia 
spaces as well.

Using Morita equivalence we can give a more precise description of
$K_G(I(\Psi))$. If $m \in \Psi$ is any element and $Z =Z_G(m)$ is the
centralizer of $m$ in $G$, then $K_G(I(\Psi)) = K_Z(X^m)$ and
$A^*_G(I(\Psi)) = A^*_Z(X^m)$. Similarly,
 if 
 $\Phi\subset G^\ell$ 
 is a diagonal conjugacy class and 
$(m_1, \ldots , m_{\ell}) \in
\Phi$
and $Z = \bigcap_{i=1}^{\ell} Z_G(m_i)$, then $K_G(\Imult{l}(\Phi)) =
K_Z(X^{m_1, \ldots , m_{\ell}})$ and $A^*_G(\Imult{l}(\Phi))= A^*_Z(X^{m_1,\ldots
  , m_{\ell}})$. 

\begin{df}
  Define a map $L \colon K_G(X) \rTo K_G(\IGX )\otimes \Q$, 
  as follows. For each conjugacy class $\Psi \subset G$ and each $V\in  K_G(X)$
  let
  $L(\Psi)(V)$ be the class in $K_Z(I(\Psi))$ which is Morita equivalent to
  $L(g)(V|_{X^g}) \in K_Z(X^g)_\Q$. Here $g$ is any element of $\Psi$,
  and $Z = Z_G(g)$ is the centralizer of $g \in G$. The class $L(V)$
  is the class whose restriction to $I(\Psi)$ is $L(\Psi)(V)$.
\end{df}
The proof of \cite[Lm.~5.4]{EJK:10} shows that 
$L(\Psi)(V)$ 
(and thus $L(V)$)
is independent of the choice of $g \in \Psi$. 

\begin{df} \label{def.twistedpullback}
If the diagonal conjugacy class $\Phi\subset G^\ell$ is represented by $(g_1,\ldots,g_\ell)$ such that $\prod_{i=1}^\ell g_i = 1$, then we define $V(\Phi)$ to be the class in $K_G(\Imult{l}_G X)$  which is Morita equivalent to 
$V({\mathbf g})$, where ${\mathbf g} = (g_1, \ldots , g_{\ell})$ is any 
element of $\Phi$. Again $V(\Phi)$ is independent of the choice
of representative ${\mathbf g} \in \Phi$.
\end{df}

\begin{df}
  Identify $\Imult{\ell}_G X$ with the closed and open subset of
  $\Imult{\ell+1}_G X$ 
  consisting of tuples 
  $\{(g_1, \ldots, g_{\ell+1},x) | g_1
  g_2 \ldots g_{\ell+1} =1\}$.  If $V \in K_G(X)$, let  $LR(V)
  \in 
  K_G(\Imult{\ell}_G X)$ 
  be the class whose restriction to
  $\Imult{l+1}(\Phi)$ is $V(\Phi)$, where 
the diagonal conjugacy class $\Phi   \in  G^{\ell+1}$ is represented by a tuple
    $(g_1, \ldots , g_{\ell+1})$ satisfying
$g_1 \cdots g_{\ell+1} = 1$.
\end{df}

\subsection{Orbifold products and the orbifold Chern character}

In this section we briefly review the construction and properties of orbifold products and 
orbifold
Chern characters because they serve as a model for what we will do later.

\begin{df}
For $i\in\{1,2,3\}$, let $e_i:\II_G X \rTo I_G X$ be the
evaluation morphism taking $(m_1,m_2,m_3, x)\mapsto (m_i,x)$ and let $\mu:\II_G X \rTo I_G X$ be the morphism taking $(m_1,m_2,m_3, x)\mapsto (m_1m_2,x)=(m_3^{-1},x)$. 
\end{df}

\begin{df}
Let $\T$ be the equivariant bundle on $X$ corresponding to the tangent bundle
of $\ix$,
which satisfies $\T = TX - \g$ in $K_G(X)$, where $\g$ is the Lie algebra of $G$.
\end{df}
\begin{df}[\protect{\cite{EJK:10,JKK:07,Kau:10a}}] \label{df:OrbifProd}
The \emph{orbifold product} on $K_G(\IGX)$ and $A_G^*(\IGX)$ is defined as
\begin{equation}\label{eq:OrbifoldProduct}
x \ti y := \mu_*(e_1^*x \cdot e_2^*y \cdot \euler(\Lr{\T})),
\end{equation}
both for $x,y \in K_G(I_G X)$ and for $x,y \in A^*_G(\IGX)$.
\end{df}

\begin{df}\label{df:S}
We define the element $\Sm := L(\T)$ in $K_G(\IGX)_\nq$ to be the logarithmic trace of $\T$,
 that is, for each $m$ in $G$, 
we define $\Sm_m$ in  $K_{Z_G(m)}(X^m)_\nq$ by 
\begin{equation*}
\Sm_m := L(m)(\T).
\end{equation*}
The rank of $\Sm$ is a $\nq$-valued, locally constant function on
$I\ix=[\IGX/G]$ called the \emph{age}.
\end{df}
\begin{rem}
If the age of a connected component 
$[U/G]$
of 
$I\ix$
is zero, then $[U/G]$ must be a
connected component of 
$\ix \subset I\ix$\!.
\end{rem}
\begin{rem}\label{rem:PanderToRalph}
In \cite{Kau:10a}  the classes $\Sm$, manifestations of what physicists call \emph{twist fields}, were interpreted in terms of cocycles which were then used to define stringy products. Our construction may be regarded as a realization of this procedure.
\end{rem}
\begin{df}
Given an element $x$ in 
$A^*_G(I_G X)$ with ordinary Chow grading 
$\deg x$, the \emph{orbifold 
degree (or grading) of $x$} is, like the ordinary Chow grading,
constant on each component 
 $U$ of $\IGX$ 
corresponding to a connected component of $[U/G]$ of $[\IGX/G]$.  On such 
a
component $U$ we define
it to be the non-negative rational number
\begin{equation}\label{eq:OrbGrading}
\res{\odeg x}{U} = \res{\deg x}{U} + 
\age([U/G]).
\end{equation}
The induced grading on the group  $A_G^*(\IGX)$ consists of summands $A_G^{\bra{q}}(\IGX)$ of all elements with orbifold degree $q$.
\end{df}

\begin{thm}[\cite{JKK:07,EJK:10}]
The equivariant Chow group $(A_G^*(\IGX),\ti, \odeg)$ 
is a 
$\nq^C$-graded, 
commutative ring with unity $\one$, where $\one$ is the identity element in $A^*_G(X) = A^*_G(X^1)\subseteq A^*_G(\IGX)$ and $C$ is the number of connected components of 
$[\IGX/G]$.

Equivariant K-theory $(K_G(\IGX),\ti)$ is a commutative ring with unity $\one$, where $\one:=\co_X$ is the structure sheaf of $X=X^1\subset \IGX$.
\end{thm}

\begin{df}
The \emph{orbifold Chern character} $\oCh:K_G(\IGX)\rTo A_G^*(\IGX)_\nq$
is defined by the equation 
\begin{equation*}
\oCh(\cf) := \Ch(\cf)\cdot\Td(-\Sm)
\end{equation*}
for all $\cf \in K_G(\IGX)$, where $\Td$ is the usual Todd class. Moreover, for all $\alpha\in \nq$ we
define $\oCh^\alpha(\cf)$ by the equation
\begin{equation*}
\oCh(\cf) = \sum_{\alpha\in\nq} \oCh^\alpha(\cf),
\end{equation*}
where each $\oCh^\alpha(\cf)$ belongs to $A^{\bra{\alpha}}_G(\IGX)$.

The \emph{orbifold virtual rank (or orbifold augmentation)} is 
$\oCh^0:K_G(\IGX)\rTo 
A^{\bra{0}}_G(\IGX)_\nq
$.
\end{df}

\begin{thm}[\cite{EJK:10,JKK:07}]
The orbifold Chern character $$\oCh:(K_G(\IGX),\ti) \rTo (A_G^*(\IGX)_\nq,\ti)$$ is a ring homomorphism.  

In particular, if $[U/G]$ is a connected component of $[\IGX/G]$,  
then the virtual rank homomorphism restricted to the component $[U/G]$ gives a homomorphism $\oCh^0:K_G(U)\rTo 
A_G^0(U) _\nq
= \nq$, satisfying
\begin{equation*}
\oCh^0(\cf) = 
\begin{cases} 
0 & \mathrm{if} 
\age([U/G])
> 0 \\
\Ch^0(\cf)  & \mathrm{if} 
\age([U/G]) 
= 0. \\
\end{cases}
\end{equation*}
for any $\cf \in K_G(U)$.
\end{thm}

\section{Inertial products, Chern characters, and inertial pairs}

In this section we generalize the ideas of orbifold cohomology, obstruction bundles, 
orbifold grading
and 
the orbifold
Chern character, by defining \emph{inertial products} on $K_G(\IGX )$ and $A^*_G(\IGX )$ using inertial bundles on 
$\II_G X$.  
We further define a rational grading and a Chern character ring homomorphism  via \emph{Chern-compatible  
classes of $K_G(\IGX)_\nq$}. 

The original example of an associative bundle is the obstruction bundle 
$\R = LR(\T)$ of orbifold cohomology, and the original example of a Chern-compatible  class 
is the logarithmic trace $\cs$ of $\T$, as described in Definition~\ref{df:S}.

We show below that there are many \emph{inertial pairs} of associative inertial bundles on $\II_G X$ with Chern-compatible elements on $\IGX$, and hence there are many associative inertial products on $K_G(\IGX )$ and $A^*_G(\IGX )$ with rational gradings and Chern character ring homomorphisms.

\subsection{Associative bundles and inertial products}
We recall the following definition from \cite{EJK:10}.  
It should be noted that a similar formalism also appeared  in the paper \cite{Kau:10a}, but we will use the notation of \cite{EJK:10}.
\begin{df}
Given a class $c \in A^*_G(\II_G X)$ (resp.~$K_G(\II_G X)$), we define the \emph{inertial product 
with respect to $c$}  
to be \begin{equation} \label{eq.inertialprod}
x \star_c y := \mu_*(e_1^*x \cdot e_2^*y \cdot c),
\end{equation}
where $x,y \in A^*_G(\IGX )$ (resp.~$K_G(\IGX )$).
\end{df}
Given a vector bundle ${\R}$ on $\II_G X$ we define inertial products on $A^*_G(\IGX )$ and $K_G(\IGX )$ via
formula \eqref{eq.inertialprod}, where $c = \euler({\R})$ is the
Euler class of the bundle ${\R}$.
\begin{df}
We say that ${\R}$ is an {\em associative bundle} on $\II_G X$  if the $\star_{\euler({\R})}$ products on both $A^*_G(\IGX )$ and $K_G(\IGX )$  are commutative and associative with identity 
$\one$, 
where
$\one$
is the identity class in $A^*_G(X)$ (resp.~$K_G(X)$),
viewed as a summand in $A^*_G(\IGX )$ (resp.~$K_G(\IGX )$). 
\end{df}
\begin{prop}
A  sufficient 
condition for $\star_{\euler(\R)}$ to be commutative with identity 
${\one}$
is that 
the following conditions be satisfied.
\begin{enumerate}
\item \begin{equation} \label{eq.identity}
{\R}|_{I(\Phi)} = {\mathscr O}
\end{equation} for every conjugacy class
$\Phi\subset G \times G$ such that $e_1(\Phi) = 1$ or $e_2(\Phi) =1$.

\item \begin{equation} \label{eq.commutative}
i^* {\R} = {\R},
\end{equation}
where $i \colon \II_G X \rTo \II_G X$ is the 
isomorphism
 $i(m_1,m_2,x) = (m_1m_2m_1^{-1},m_1,x)$.

\end{enumerate}
\end{prop}
\begin{proof}
This is almost just a restatement of Propositions 3.7--3.9 in \cite{EJK:10}.
However, we note that in Proposition~3.9 of \cite{EJK:10} there is a slight error---that proposition incorrectly stated that the map $i\colon\Itwo_G X \rTo \Itwo_G X$ was the map induced by the naive involution $(m_1,m_2) \mapsto (m_2,m_1)$, rather than the correct ``braiding map'' $(m_1,m_2,x) \mapsto (m_1m_2m_1^{-1},m_1,x)$.    
\end{proof}
A sufficient condition for associativity is also given in
\cite{EJK:10}.  To state the condition requires some notation which we
recall from that paper.  
Let $(m_1, m_2,m_3) \in G^3$ such that $m_1m_2m_3 = 1$, and let
$\Phi_{1,2,3}\subset G^3$ be its diagonal conjugacy class. 
Let $\Phi_{12,3}$
be the conjugacy class of $(m_1m_2,m_3)$ and $\Phi_{1,23}$ the
conjugacy class of $(m_1,m_2m_3)$.  Let $\Phi_{i,j}$ be the conjugacy
class of the pair $(m_i,m_j)$ with $i < j$. Finally let $\Psi_{123}$
be the conjugacy class of $m_1m_2m_3$; let $\Psi_{ij}$
be the conjugacy class of $m_im_j$; and let $\Psi_{i}$
be the conjugacy class of $m_i$.  There are evaluation maps
$e_1\colon \II(\Phi_{a,b})\rTo I(\Psi_{a}),$ $e_2\colon \II(\Phi_{a,b})\rTo I(\Psi_{b}),$ 
$e_{i,j} \colon {\Imult3}(\Phi_{1,2,3}) \rTo \II(\Psi_{i,j})$ ,
and 
the composition maps $\mu_{12,3} \colon {\Imult3}(\Phi_{1,2,3}) \rTo
\II(\Phi_{12,3}),$  
and $\mu_{1,23} \colon {\Imult3}(\Phi_{1,2,3}) \rTo\II(\Phi_{1,23})$.
The various maps we have defined are related by the
following Cartesian diagrams of l.c.i.~morphisms:

\begin{equation} \label{diag.excess12}
\begin{diagram}
{\Imult3}(\Phi_{1,2,3}) &\rTo^{e_{1,2}} & \II(\Phi_{1,2}) \\
\dTo^{\mu_{12,3}}  & &  \dTo_{\mu}\\
\II(\Phi_{12,3}) & \rTo^{e_1} & I(\Psi_{12})
\end{diagram}
\qquad 
\begin{diagram}
{\Imult3}(\Phi_{1,2,3}) &\rTo^{e_{2,3}}  & \II(\Phi_{2,3}) \\
\dTo^{\mu_{1,23}} & & \dTo_{\mu}\\
\II(\Phi_{1,23}) & \rTo^{e_2} & I(\Psi_{23})
\end{diagram}
\end{equation}
Let $E_{1,2}$ and $E_{2,3}$ be the respective excess normal bundles of the two diagrams \eqref{diag.excess12}.
\begin{prop}
Let ${\R}$ be a vector bundle on $\II_GX$ satisfying \eqref{eq.identity}
and \eqref{eq.commutative}. A sufficient condition for ${\R}$ to be an associative bundle is if
\begin{equation} \label{eq.neccasso} 
e_{1,2}^*{\R} + \mu_{12,3}^*{\R} + E_{1,2} = 
e_{2,3}^*{\R} + \mu_{1,23}^*{\R} + E_{2,3}
\end{equation}
in $K_G({\Imult3}X)$
\end{prop}
\begin{proof}
This follows from the proof of Proposition 3.12 \cite{EJK:10}, since the Euler
class takes a sum of bundles to a product of Euler classes.
\end{proof}
In practice, the only way we have to show that a bundle ${\R}$
is associative is to show that it satisfies the identity
\eqref{eq.neccasso}. This leads to our next definition.
\begin{df}
A bundle ${\R}$ is \emph{strongly associative} if it satisfies the identities 
\eqref{eq.identity}, 
\eqref{eq.commutative}, and 
\eqref{eq.neccasso}.
\end{df}

\subsection{Chern characters, age, and inertial pairs}

In many cases one can define a Chern character $K_G(\IGX )_\Q \rTo A^*_G(\IGX
)_\Q$ which is a \emph{ring homomorphism} with respect to the inertial
product.  To do this, however, we need to define a \emph{Chern compatible
  class $\cs \in K_G(\IGX)$}.  As an added bonus, such a class will also
allow us to define a new grading on $A^*_G(\IGX)$ compatible with the
inertial product and  analogous to the 
orbifold grading
of orbifold cohomology.  

\begin{df}
Let ${\R}$ be an associative bundle on $\II_G X$. A non-negative
class 
${\cs}\in K_G(\IGX )_\nq$ 
is called ${\R}$-Chern compatible if the map 
$$\inCh \colon K_G(\IGX )_\Q \rTo A^*_G(\IGX )_\Q,$$ defined by 
$$\inCh(V) = \Ch(V)
\cdot
\Td(-{\cs})$$ 
is a ring homomorphism with respect to the 
${\R}$-inertial products on $K_G(\IGX )$ and $A^*_G(\IGX )$. 
\end{df}

\begin{rem}
  Again, the original example of a Chern compatible class is the class
  ${\cs}$ defined in \cite{JKK:07}, but we will we see other
  examples below.\end{rem}
\begin{prop}
  If ${\R}$ is an associative vector bundle on
  $\II_G X$, then a non-negative class ${\cs} \in
  K_G(\IGX )_{Q}$ is ${\R}$-Chern compatible if the following identity
  holds in $K_G(\II_G X)$
\begin{equation} \label{eq.chernidentity}
{\R} = e_1^* {\cs} + e_2^* {\cs} - \mu^*{\cs} + T_{\mu}.
\end{equation}
\end{prop}
\begin{proof}
This follows from the same formal argument used in the proof of \cite[Theorem 7.3]{EJK:10}.
\end{proof}

\begin{df}
A class ${\cs} \in K_G(\IGX )_\Q$ is {\em strongly}   ${\R}$-Chern compatible if it
satisfies 
Equation
\eqref{eq.chernidentity}. 
  
  A pair $({\R}, {\cs})$ is
  an {\em inertial pair} if ${\R}$ is a strongly associative
  bundle and ${\cs}$ is ${\R}$-strongly Chern
  compatible. 
\end{df}
\begin{df}
  We define the ${\cs}$-{\em age} on a 
connected
component 
$[U/G]$
of
  $I\ix $ 
to be the rational rank of ${\cs}$ on the component 
$[U/G]$:
  $$\age_{\cs}(
[U/G]
) = \rk(\cs)_{
[U/G]
}.$$
   We define the ${\cs}$-degree of an element $x \in
  A^*_G(\IGX )$ 
  on 
such a component
$U$ of $\IGX$
   to be 
           $$\res{\deg_\cs x}{U}  = \res{\deg x}{U} + \age_{\cs}(
[U/G]
),$$
   where $\deg x$ is
  the degree with respect to the usual grading by codimension on
  $A^*_G(\IGX )$.
Similarly, if $\cf$ in $K_G(\IGX)$ 
is an element supported on $U$, 
then its $\cs$\!\!-degree 
is
\[
\deg_\cs \cf = \age_{\cs}(U)\bmod\nz.
\]
This yields a $\nq/\nz$-grading of the group $K_G(\IGX)$.

\end{df}

\begin{prop}
  If ${\R}$ is an associative vector bundle on
  $\II_G X$ and ${\cs} \in K_G(\IGX )_\Q$ is strongly
  ${\R}$-Chern compatible, then the ${\R}$-inertial
  products on $A^*_G(\IGX)$ and $K_G(\IGX)$ respect the $\cs$\!\!-degrees. Furthermore, the inertial Chern character homomorphism $\inCh:K_G(\IGX)\rTo A_G^*(\IGX)_\nq$ preserves the $\cs$\!\!-degree modulo $\nz$.
\end{prop}
\begin{proof}
If $x, y \in A^*_G(\IGX )$, then the formula
$$
x \star_{\euler({\R})}  y = \mu_*\left(e_1^*x \cdot e_2^* y \cdot 
\euler({\R})\right)
$$
implies that 
$\deg(x \star_{\euler({\R})}
y) = \deg x + \deg y + \rk {\R} + \rk T_\mu.$
Since ${\cs}$ is strongly ${\R}$-Chern compatible,
we know that
${\R} = e_1^*{\cs} + e_2^*{\cs} -\mu^*{\cs}
+ T_{\mu}$. Comparing ranks shows that the ${\cs}$-degree of
$x\star_{\euler({\R})} y$ 
is the sum of the ${\cs}$-degrees
of $x$ and $y$.
The proof for $K_G(\IGX)$ follows from the fact that $\rk{\R}$ and $\rk T_\mu$ are integers. Finally, $\inCh$ preserves the ${\cs}$-degree mod $\nz$ since if $\cf$ in $K_G(\IGX)$ is supported on $U$, where $[U/G]$ is a connected component of $[\IGX/G]$, then so is its inertial Chern character.
\end{proof}

\begin{df}
Let $A_G^{\bra{q}}(\IGX)$ be the subspace in $A_G^*(\IGX)$ of elements with an $\cs$\!\!-degree of $q$.
\end{df}

\begin{df}
Given a class ${\cs}\in K_G(\IGX )_\nq$, 
the restricted homomorphism  
$\inCh^0:K_G(\IGX)\rTo A_G^{\bra{0}}(\IGX)$ is called the
\emph{inertial virtual rank (or inertial augmentation)} 
for $\cs$\!\!. 
\end{df}

\begin{df}
    An inertial pair $({\R}, {\cs})$ is
  called {\em Gorenstein} if ${\cs}$ has integral virtual rank
and {\em strongly Gorenstein} if ${\cs}$ is represented by a vector
  bundle.

The Deligne-Mumford stack ${\stack X} = [X/G]$ is {\em strongly Gorenstein}
if the inertial pair $({\R =LR(\T)},{\cs})$ associated to the orbifold product (as in Definitions~\ref{df:OrbifProd} and \ref{df:S}) is strongly Gorenstein.
\end{df}

\section{Inertial pairs associated to vector bundles}
In this section we show how, for each choice of $G$-equivariant bundle
$V$ on $X$, we can use the methods of \cite{EJK:10} to define two new
inertial pairs 
$(\crr^+V, \cs^+V)$ and $(\crr^-V,\cs^-V)$. 
We thus obtain corresponding inertial
products and Chern characters.  We denote the corresponding products
associated to a vector bundle $V$ as the $\star_{V^+}$ and
$\star_{V^-}$ products. The $\star_{V^+}$ 
product
can be interpreted as an
orbifold product on the total space of $V$ while the $\star_{V^-}$
product on the Chow ring is 
a sign twist of the $\star_{(V^*)^+}$
product. 
Moreover, the two products induce isomorphic ring structures
on $A^*(I\ix)\otimes {\mathbb C}$.
We prove that if $V = {\mathbb T}$ is the tangent bundle to
$\ix = [X/G]$, then the $\star_{V^-}$ product agrees with the virtual
orbifold product defined by \cite{GLSUX:07}.

To define the inertial pairs associated to a vector bundle, 
we introduce a variant of
the logarithmic restriction introduced in \cite{EJK:10}. We begin with a simple proposition.
\begin{prop} \label{prop.integral+-}
Let $G$ be an algebraic group acting on a variety $X$ and suppose that
$g_1, g_2$ lie in a common compact subgroup. Let 
$Z = Z_G(g_1,g_2)$
be the centralizer of $g_1$ and $g_2$ in $G$.

The virtual bundles
\begin{equation} \label{eq.vplus}
V^+(g_1,g_2) = L(g_1)(V|_{X^{g_1,g_2}}) + L(g_2)(V|_{X^{g_1,g_2}}) 
- L(g_1g_2)(V|_{X^{g_1,g_2}}) 
\end{equation}
and 
\begin{equation} \label{eq.vminus}
V^-(g_1,g_2) = L(g_1^{-1})(V|_{X^{g_1,g_2}}) + L(g_2^{-1})(V|_{X^{g_1,g_2}}) 
- L(g_2^{-1}g_1^{-1})(V|_{X^{g_1,g_2}}) 
\end{equation}
are represented by non-negative integral elements in 
$K_Z(X^{g_1,g_2})$.
\end{prop}
\begin{proof}
  Since $X^{g} = X^{g^{-1}}$ and $V^-(g_1,g_2) = V^+(g_2^{-1}, g_1^{-1})$,
it suffices to show $V^+(g_1,g_2)$ is
  represented by an non-negative integral element of
  $K_Z(X^{g_1,g_2})$. Let $g_3 = (g_1g_2)^{-1}$.  The identity $L(g)(V)
    + L(g^{-1})(V) = V - V^{g}$ implies that we can rewrite
    \eqref{eq.vplus} as
\begin{equation*}
V^+(g_1,g_2) = L(g_1)(V|_{X^{g_1,g_2}}) + L(g_2)(V|_{X^{g_1,g_2}}) 
+ L(g_3)(V|_{X^{g_1,g_2}}) - V + V^{g_1,g_2} +V^{g_1g_2} -V^{g_1,g_2}.
\end{equation*}
Since $g_1g_2g_3 = 1$, by Proposition~\ref{prop.logrestriction}  
\cite[Prop 4.6]{EJK:10} the sum
\begin{equation*}
L(g_1,g_2,g_3)(V) = L(g_1)(V|_{X^{g_1,g_2}}) + L(g_2)(V|_{X^{g_1,g_2}}) 
+ L(g_3)(V|_{X^{g_1,g_2}}) - V + V^{g_1,g_2}
\end{equation*}
is represented by a non-negative integral element of 
$K^Z(X_{g_1,g_2})$.
Hence 
\begin{equation*}
V^+(g_1,g_2) = L(g_1,g_2,g_3)(V) + V^{g_1g_2} - V^{g_1,g_2}
\end{equation*}
is represented by a non-negative integral element of 
$K_Z(X^{g_1,g_2})$.
\end{proof}
Let $\Phi \subset G \times G$ be a diagonal conjugacy class. 
As in \cite{EJK:10} we may identify 
$K_G(\II(\Phi))$ with  $K_{Z_G(g_1,g_2)}(X^{g_1,g_2})$ for any $(g_1,g_2) \in \Phi$.
Thanks to Proposition \ref{prop.integral+-} we can define non-negative classes
$V^+(\Phi)$ and $V^-(\Phi)$ in  $K_G(\II(\Phi))$. The argument used in the proof of
\cite[Lemma 5.4]{EJK:10} shows that the definitions of $V^+(\Phi)$ and
$V^-(\Phi)$ are independent of the choice of $(g_1,g_2) \in 
G^2
$.
Thus we can make the following definition.
\begin{df}
Define classes $\ar^+V$ and $\ar^-V$ in $K_G(\II_G X)$ by setting the component of $\ar^+V$
(resp.~$\ar^-V$) in $K_G(\II(\Phi))$ to be $V^+(\Phi)$ (resp.~$V^-(\Phi)$).
Similarly, we define classes 
${\as}^\pm V \in K_G(\IGX )_\Q$ by setting
the restriction of 
${\as}^\pm V$ 
to a summand $K_G(I(\Psi))$ of
$K_G(\IGX )$ to be the class Morita equivalent  to $L(g^{\pm 1})(V)
\in 
K_{Z_G(g)}(X^g)$, 
where $g \in \Psi$ is any element.
\end{df}

\begin{thm}
For any $G$-equivariant vector bundle $V$ on $X$, the pairs 
$
(\crr^+ V,\cs^+ V) =
(\Lr{\T} + {\ar^+V}, {\cs}\T + {\as}^+V)$
and 
$
(\crr^-V,\cs^-V) = 
(\Lr{\T} + {\ar^-V}, {\cs}\T + {\as}^-V)$
are inertial pairs. Hence they define associative inertial products with a
Chern character homomorphism.
\end{thm}
\begin{proof}
Since $\Lr{\T} = e_1^*{\cs}\T + e_2^*{\cs}\T - \mu^*{\cs}\T + T_\mu$
and $\ar^+V = e_1^*{\as}^+V + e_2^*{\as}^+V - \mu^*{\as}^+V$,
it follows that ${\cs^+V}$ is strongly ${\crr^+V}$-Chern compatible.

To complete the proof we must show that $\Lr{\T} + {\ar^+V}$ is a strongly associative bundle. From their definitions we know that $\Lr{\T}$ and $\ar^+V$ satisfy
the identities \eqref{eq.identity} and \eqref{eq.commutative}. We also know that
$\Lr{\T}$ satisfies \eqref{eq.neccasso}. Thus to prove that 
$\Lr{\T} + {\ar^+V}$, it suffices to show that $\ar^+V$ satisfies the ``cocycle'' condition
\begin{equation} \label{eq.assocsuffice}
e_{1,2}^*\ar^+V + \mu_{12,3}^* \ar^+V = e_{2,3}^* \ar^+V + \mu_{1,23}^* \ar^+V.
\end{equation}
Now \eqref{eq.assocsuffice} follows 
from the following identity of bundles restricted to $X^{m_1,m_2,m_3}$:
\begin{equation} \label{eq.11easiest}
V^+(m_1,m_2) + V^+(m_1m_2,m_3) = V^+(m_2,m_3) + V^+(m_1,m_2m_3).
\end{equation}
Equation \eqref{eq.11easiest} is a formal consequence of the
definition of the bundles $V^+$. The result with $\ar^+V$ and ${\as}^+V$
replaced by $\ar^-V$ and ${\as}^-V$, respectively, is analogous.
\end{proof}
\subsection{Geometric interpretation of the $\star_{V^+}$ product} \label{subsec.geometric}
The $\star_{V^+}$ has a relatively direct interpretation in terms of an orbifold product on the total space of the vector bundle $V \rTo X$.

\begin{lm}
  Given a $G$-equivariant vector bundle $\pi:V\rTo X$, the inertia
 space $I_GV$ is a vector bundle (of non-constant rank) on $\IGX$
  with structure map $I\pi:I_GV \rTo \IGX$.
\end{lm}
\begin{proof}
Let $\Psi \subset G$ be a conjugacy class. Denote by $I_X(\Psi) \subset \IGX $
the component of $\IGX $ defined by $\{(g,x) | gx = x, \ g\in\Psi \}$. 
For any morphism $V \rTo X$ and any conjugacy class $\Psi \in G$, if 
$I_X(\Psi) = \emptyset$, then  $I_V(\Psi)$ is also empty.
Thus it suffices to show that $I_V(\Psi)$ is a vector bundle over $I_X(\Psi)$
for every conjugacy class $\Psi \subset G$ with $I_X(\Psi) \neq \emptyset$.
Given $g \in \Psi$ the identification $I_X(\Psi) = G \times_{Z_G(g)} X^g$
reduces the problem to showing that for $g \in G$ the fixed locus
$V^g$ is a $Z_G(g)$-equivariant vector bundle over $X^g$. Since the map $V \rTo X$
is $G$-equivariant, the map $V^g \rTo X$ has image $X^g$. The fiber over
a point $ x\in X^g$ is just $(V_x)^g$, where $V_x$ is the fiber of $V \rTo X$
at $x$.
\end{proof}

Since $I_G V \rTo I_G X$ is a vector bundle, the pullback maps
$$(I\pi)^* \colon K_G(I_GV ) \rTo K_G(I_G X) \quad \text{ and } \quad 
(I\pi)^* \colon A^*_G(I_GV) \rTo A^*_G(I_G X)$$ are isomorphisms.
Both isomorphisms are compatible with the 
ordinary
products on K-theory and equivariant Chow groups.

\begin{thm}\label{prop.star+geo}
For $x,y \in A^*_G(\IGX )$ or $x,y \in K_G(X)$, we have 
\begin{equation} \label{eq.geomv+}
x \star_{V^+} y = (Is)^*\left((I\pi)^* x \star (I\pi)^*y \right),
\end{equation}
where $\star$ is the usual orbifold product on the total space of the $G$-equivariant vector bundle $V \rTo X$ and $Is^*$
is the Gysin map which is inverse to to $I\pi^*$.
\end{thm}
\begin{proof}
  We give the proof only in equivariant Chow theory---the proof in equivariant
  K-theory is essentially identical.  We compare the two sides of
  \eqref{eq.geomv+}. If $\Psi_1, \Psi_2, \Psi_3 \subset G$ are
  conjugacy classes and $x \in A^*_G(I_X(\Psi_1)), y \in
  A^*_G(I_X(\Psi_2))$, then the contribution of $x \star_{V^+} y$
to $A^*_G(I_X(\Psi_3))$ is 
\begin{equation} \label{eq.lhsv+}
\sum_{\Phi_{1,2}} \mu_*\left(e_1^*x \cdot 
e_2^*y
\cdot \euler(\Lr{\T} + \ar^+V)
\right),
\end{equation}
where the sum is over all conjugacy classes $\Phi_{1,2} \subset G \times G$ 
satisfying
\begin{equation*} 
e_1(\Phi_{1,2}) = \Psi_1, e_2(\Phi_{1,2}) = \Psi_2, 
\mu(\Phi_{1,2}) = \Psi_3.
\end{equation*}
Since the class of tangent bundle of $V$ equals $TX + V$, the tangent bundle to the stack $[V/G]$ is $TX +V - {\mathfrak g} = \T + V$.
Thus, the contribution of the right-hand side of \eqref{eq.geomv+} is the sum
\begin{equation}\label{eq.rhsv+}
\sum_{\Phi_{1,2}} Is^*
\left(\mu_{V*}(I\pi)^*
(e_1^*x \cdot 
e_2^*y 
\cdot \euler(\Lr{\T} +\Lr{V}))\right),
\end{equation}
where the map $\mu_V$ in \eqref{eq.rhsv+} is understood to be the
multiplication map 
$\II_G V \rTo I_G V$.  
If $\Phi$ is a conjugacy
class in $G \times G$ with $\mu(\Phi) = \Psi$, then the multiplication
map $\mu_V \colon I_V(\Phi) \rTo I_V(\Psi)$ factors through the inclusion
\begin{equation} \label{eq.factor}
I_V(\Phi) \rInto\mu^*I_V(\Psi)  \rTo^{I\pi^*\mu} I_V(\Psi),
\end{equation}
and we have the following diagram, with a Cartesian square on the right.
$$\begin{diagram}
I_V(\Phi)  & \rInto & \mu^*I_V(\Psi) & \rTo^{I\pi^*\mu} & I_V(\Psi)\\
& \rdTo^{I\pi} &\dTo & & \dTo^{I\pi}\\
& & I(\Phi) & \rTo^{\mu} & I(\Psi)
\end{diagram}
$$
The normal bundle to the inclusion $I_V(\Phi)
\rInto \mu^*(I_V(\Psi))$ is 
the pull-back of 
the bundle $V_\Psi/V_\Phi$ on $\II_X(\Phi)$, where $V_\Phi \subset
\res{V}{I_X(\Psi)}$ is the subbundle whose fiber over a point $(g,x)$ is
the subspace $V^{g}$, and the fiber of $V_\Phi$ over a point $(g_1,g_2,x)$
is the subspace $V^{g_1,g_2} \subset V$. Using this information about the normal bundle we can rewrite
\eqref{eq.rhsv+} as 
\begin{equation} \label{eq.rhsv+massage}
\mu_*\left(e_1^*x \cdot e_2^*y 
 \cdot
 \euler(\Lr{\T+V} + V_\Psi - V_\Phi\right)
).
\end{equation}
Finally, expression \eqref{eq.rhsv+massage} can be identified with \eqref{eq.lhsv+} by observing that if $g_1,g_2 \in G$, then
\begin{eqnarray*} L(g_1)(V) + L(g_2)(V) + L((g_1g_2)^{-1})(V) + V - V^{g_1,g_2}
 +V^{g_1,g_2} - V^{g_1g_2} & =& \\
L(g_1)(V) + L(g_2)(V) -L(g_1g_2)(V). &  &
\end{eqnarray*}
\end{proof}
\subsection{Geometric interpretation of the $\star_{V^-}$ product}\label{subsect:geomneg}
The $\star_{V^-}$ product does not generally correspond to an
orbifold product on a bundle. However, we will show that, after tensoring with $\nc$,
the inertial Chow (or cohomology) ring with the $\star_{V^-}$ product is 
isomorphic to the inertial Chow (or cohomology) ring coming
from the total space of the dual bundle.
The latter is isomorphic to the orbifold Chow (or cohomology) ring of the total space of the dual bundle.
\begin{df}
Given a vector bundle $V$ on a quotient stack $\ix = [X/G]$, we
define an automorphism $\Theta_V$ of $A^*(I\ix) \otimes \nc$ as follows. If $x_\Psi$
is supported on a component $I(\Psi)$ of $I\ix$ corresponding to a conjugacy class $\Psi \subset G$ then we set $\Theta_V(x_\Psi) = e^{i \pi a_\Psi} x_\Psi$,
where $a_\Psi$ is the virtual rank of the logarithmic trace 
$L(g^{-1})(V)$
for any
representative element $g \in \Psi$.
The same formula defines an automorphism of $H^*(I\ix, \nc)$. 
\end{df}

\begin{thm}\label{thm.star-geo}
For $x,y \in A^*_G(\IGX )$ we have
\begin{equation} 
\label{eq.geomv-}
\begin{array}{rcl}
x \star_{V^-} y & = &\pm (Is)^*\left((I\pi)^* x \star (I\pi)^*y \right)\\
&= & \pm x \star_{V^{*+}}y,
\end{array}
\end{equation}
where $\star$ is the usual orbifold product on the total space of the
$G$-equivariant vector bundle $V^* \rTo X$, and $Is^*$ is the Gysin map
which is inverse to to $I\pi^*$,
and the sign $\pm$ is $(-1)^{a_{\Psi_1}+a_{\Psi_2}-a_{\Psi_{12}}}$ 
where ${a_{\Psi_1}+a_{\Psi_2}-a_{\Psi_{12}}}$ is a non-negative integer.
Moreover, if we tensor with ${\nc}$, then
we have the identity
\begin{equation} \label{eq.star-dual}
\Theta_V(x \star_{V^-} y)  = \Theta_V(x) \star_{V^{*+}} \Theta_V( y).
\end{equation}
\end{thm}
\begin{rem}\label{rem:disc-tors}
The $\pm$ sign appearing in the previous theorem is an example of discrete torsion. Similar signs appear in the work of \cite{FaGo:03}.
\end{rem}
\begin{proof}
Observe that if $g \in G$ acts on a representation $V$ of 
$Z= Z_G(g)$ 
with weights
$e^{i \theta_1}, \ldots , e^{i \theta_r}$, then $g$ naturally acts on $V^*$
with weights 
$e^{-i \theta_1} , \ldots , e^{-i \theta_r}$, and the $e^{i \theta_k}$ eigenspace
of $V$ is dual to the $e^{-i \theta_k}$ eigenspace of $V^*$. Hence
$L(g^{-1})(V) = L(g)(V^*)^*$ as elements of $K(X^g) \otimes \Q$. 
Thus given a pair $g_1, g_2 \in G$,
we see that 
$V^-(g_1,g_2) = \left((V^*)^+(g_1,g_2)\right)^*$ as 
a $Z_G(g_1,g_2)$-equivariant 
bundles
on $X^{g_1,g_2} = X^{g_1^{-1},g_2^{-1}}$.
Hence, $\euler(\ar^-V) = (-1)^{\rk \; \ar^-V}\euler(\ar^+V^*)$, 
so \eqref{eq.geomv-} holds.
If $x$ is supported in the component $I(\Psi_1)$ and
$y$ is supported in the component $I(\Psi_2)$, then
$x \star_{V^-} y$ is supported at components $I(\Psi_{12})$, where
$\Psi_{12}$ is a conjugacy class of $g_1 g_2$ for some $g_1 \in \Psi_1$
and $g_2 \in \Psi_2$. 

Now we have $$\Theta_V(x \star_{V^-} y) = \sum_{\Psi_{12}} e^{i \pi{a_{\Psi_{12}}}} 
(-1)^{\rk \;V^{-}(g_1,g_2)}
x \star_{V^{*+}} y,$$ 
while
$$\Theta_V(x) \star_{V^{*+}} \Theta_V(y) = \sum_{\Psi_{12}}e^{i \pi (a_{\Psi_1} +a_{\Psi_2})} x \star_{V^{*+}} y.$$
Thus Equation~\eqref{eq.star-dual} follows from the fact that
$\rk V^{-}(g_1,g_2) = a_{\Psi_1} + a_{\Psi_2} - a_{\Psi_{12}}$.
\end{proof}

\subsection{The virtual orbifold product is the $\star_{TX^-}$ product}\label{sec:virprod}
The virtual orbifold product 
was introduced in \cite{GLSUX:07}. In our context it 
 (or more precisely its algebraic analogue) 
can be defined as follows:
\begin{df}
Let $\T^{virt}$ be the class in $K_G(\II_G X)$ defined by the formula
\begin{equation}\label{eq:TvirtDef}
\T|_{\II_G X} + \T_{\II_G X} - e_1^*\T_{\IGX } - e_2^*\T_{\IGX },
\end{equation}
where $\res{\T}{\II_G X}$ refers to the pullback of the class $\T$ to $\II_G X$
via any of the three natural maps $\II_G X \rTo X$, where $\T_{\IGX }$ denotes the
tangent bundle to the stack 
$I\ix= [\IGX /G]$, 
and where $\T_{\II_G X}$ denotes the tangent bundle to the stack $I^2{\stack X}$. 
\end{df}

\begin{prop}
  The identity $\T^{virt} = \Lr{\T} + {\ar^-\T}$ holds in $K_G(\II_G X)$.
  In particular, $\T^{virt}$ is represented by a non-negative
  element of $K_G(\II_G X)$ and the $\star_{\euler{\T^{virt}}}$-product is commutative and associative. Moreover, $\cs\T + \as^-\T = \N$,  
 where $\N$ is the normal bundle of the canonical morphism $\IGX \rTo X$,
so $(\T^{virt},\N)$ is a strongly Gorenstein inertial pair.
\end{prop}
\begin{proof}
The proof follows from the identity $L(g)(\T) + L(g^{-1})(\T) = 
\T|_{X^g}
- \T_{X^g}
=\res{\N}{X^g}$.
\end{proof}
\begin{df}
  Following \cite{GLSUX:07}, we define the \emph{virtual orbifold product} to
  be the $\star_{\euler(\T^{virt})}$-product.
\end{df}
\begin{crl}
The virtual product 
$\star_{virt}$
on $A^*_G(I_GX)$ agrees up to sign with 
the $\star_{\T^{* +}}$ inertial 
product on 
$A_G^*(I_GX)$ 
induced by the cotangent bundle $\T^*$ of $\ix = [X/G]$, 
and there is an isomorphism of rings $(
A^*_G(I_GX)_\nc,
          \star_{virt})  \cong (
          A^*_G(I_GX) _\nc,
          \star_{\T^{*+}})$.
\end{crl}

\begin{rem}
While in Chow these products differ by a simple discrete torsion (see Remark~\ref{rem:disc-tors}), in K-theory the virtual product is not so easily identified with the product $\star_{\T^{* +}}$, as can be seen from the fact that the gradings do not match (discrete torsion does not change the grading).  If there is a connection in K-theory, it will have to be via something much more general, like a K-theoretic version of the matrix discrete torsion of \cite{Kau:10a}.
\end{rem}

\subsection{An example with $\Pro(1,3,3)$} \label{sec.pro13} We
illustrate the various inertial products in $K$-theory and Chow
theory with the example of the weighted projective space $\ix =
\Pro(1,3,3) = [(\A^3\smallsetminus \{0\})/{\mathbb C}^*]$ where
${\nc}^*$ acts with weights $(1,3,3)$.  The inertia $I\ix$ has
three sectors---the identity sector $\ix^1 = \ix$ and two twisted
sectors $\ix^{\omega}$ and $\ix^{\omega^{-1}}$, where $\omega = e^{2\pi
  i/3}$. Both twisted sectors are isomorphic to 
a  $\B{\mathbb \mu}_3$-gerbe
over $\Pro^1$.
The $K$-theory of each sector is a quotient of the
representation ring $R({\mathbb C}^*)$. Precisely, we have $K(\ix^1) =
\Z[\chi]/\left< (\chi-1)(\chi^3 -1)^2\right>$, and $K(\ix^\omega) =
K(\ix^{\omega^{-1}})= \Z[\chi]/\left<(\chi^3-1)^2\right>$, 
where $\chi$ is the
defining character of ${\mathbb C}^*$. The projection formula in
equivariant $K$-theory implies that any inertial product is determined
by the products $\one_{g_1} \star \one_{g_2} \in K(\ix^{g_1g_2})$, 
where $\one_g$ is the $K$-theoretic fundamental class on the
sector $\ix^g$.

The usual orbifold product is represented by the following 
symmetric
matrix:

\begin{center}
\begin{tabular}{l|ccc}
 & $\ix^1$ & $\ix^\omega$ & $\ix^{\omega^{-1}}$\\
\hline
$\ix^1$  & 1 & 1 & 1\\
$\ix^{\omega}$  & & 1 & $\euler(\chi)$\\
$\ix^{\omega^{-1}}$  & & & $\euler(\chi)$
\end{tabular}
\end{center}

The virtual and orbifold cotangent products are represented by the following matrices:

\begin{center}
\begin{tabular}{r|ccc}
 & $\ix^1$ & $\ix^\omega$ & $\ix^{\omega^{-1}}$\\
 \hline
$\ix^1$  & 1 & 1 & 1\\
$\ix^{\omega}$  & & $\euler(\chi)$ & $\euler(\chi)^2$\\
$\ix^{\omega^{-1}}$  & & & $\euler(\chi)$
\end{tabular}
\qquad
\begin{tabular}{r|ccc}
& $\ix^1$ & $\ix^\omega$ & $\ix^{\omega^{-1}}$\\
\hline
$\ix^1$ & 1 & 1 & 1\\
$\ix^{\omega}$ & & $\euler(\chi^{-1})$ & $\euler(\chi)\euler(\chi^{-1})$\\
$\ix^{\omega^{-1}}$ & & & $\euler(\chi)$
\end{tabular}
\end{center}

If we denote by $t= c_1(\ix) \in A^*(\B{\mathbb C^*})$, then the inertial
products on Chow and cohomology groups can also be represented by matrices
as above. After tensoring with ${\mathbb C}$,
the Chow groups of the sectors are $A^*(\ix^1)_\nc = {\mathbb C}[t]/\left< t^3\right>$
and $A^*(\ix^\omega)_\nc = A^*(\ix^{\omega^{-1}})_\nc = {\mathbb C}[t]/\left< t^2\right>$.
The corresponding matrices for the virtual and cotangent orbifold products are the following:

\begin{center}
\begin{tabular}{r|ccc}
 & $\ix^1$ & $\ix^\omega$ & $\ix^{\omega^{-1}}$\\
\hline
$\ix^1$ & 1 & 1 & 1\\
$\ix^{\omega}$ & & t & $t^2$\\
$\ix^{\omega^{-1}}$ & & & $t$
\end{tabular}
\qquad
\begin{tabular}{r|ccc}
 & $\ix^1$ & $\ix^\omega$ & $\ix^{\omega^{-1}}$\\
\hline
$\ix^1$ & 1 & 1 & 1\\
$\ix^{\omega}$  & & $-t$ & $-t^2$\\
$\ix^{\omega^{-1}}$  & & & $t$
\end{tabular}
\end{center}

The automorphism of 
$A^*(I\ix)_\nc$ 
which is the identity
on 
$A(\ix^1)_\nc$ 
and which acts by multiplication by $e^{2\pi i/3}$
on 
$A^*(\ix^\omega)_\nc$
and $e^{\pi i /3}$ on 
$A^*(\ix^{\omega^{-1}})_\nc$ 
defines a ring isomorphism between these products.

\section{The  localized orbifold  product on 
$K(I\ix) \otimes {\nc}$
}
If an algebraic group $G$ acts with finite stabilizer on smooth
variety $Y$, then there is a decomposition of $K_G(Y)\otimes {\mathbb C}$ as a sum of
localizations $\bigoplus_{\Psi} K_G(Y)_{\m_{\Psi}}$.  Here the sum is over
conjugacy classes $\Psi \subset G$ such that $I(\Psi) \neq \emptyset$,
and $\m_\Psi \in \Spec R(G)$ is the maximal ideal of class functions
vanishing on the conjugacy class $\Psi$. 

Given a conjugacy class 
$\Psi \subset G$ 
and a choice of $h\in\Psi$, denote the centralizer of $h$ in $G$ by $Z = Z_G(h)$.  The conjugacy class of $h$ in $Z$ is just $h$ alone, and there is a corresponding maximal ideal $\m_h \in \spec R(Z)$.  As described in \cite[\S4.3]{EdGr:05}, the localization $K_G(I(\Psi))_{\m_h}$ is a summand of the localization $K_G(I(\Psi))_{\m_\Psi}$, and this summand is independent of the choice of $h$.  This is called the \emph{central summand} of $\Psi$ and is denoted by $K_G(I(\Psi))_{c(\Psi)}$. 

Since $G$ acts with finite stabilizer, the projection 
$f_{\Psi}:I(\Psi) \rTo Y$
is a finite l.c.i.~morphism. The non-Abelian localization theorem of
\cite{EdGr:05} states that the pullback $f_\Psi^* \colon K_G(Y)\otimes
{\mathbb C} \rTo K_G(I(\Psi)) \otimes {\mathbb C}$ induces an
isomorphism between the localization of $K_G(Y)$ at $\m_\Psi$ and
the   central summand 
$K_G(I(\Psi))_{c(\Psi)}\subset K_G(I(\Psi))_{\m_\Psi}$. 
The inverse to
$f^*_\Psi$ is the map $ \alpha \mapsto f_{\Psi *} 
(\alpha 
\cdot
\euler(N_{f_\Psi})^{-1})$. If we let $f$ be the global stabilizer map $I_G Y
\rTo Y$, then, after summing over all conjugacy classes $\Psi$ in the
support of $K_G(Y) \otimes {\mathbb C}$, we obtain an isomorphism
$$f^* \colon K_G(Y) \otimes {\mathbb C} \rTo K_G(I_GY)_c,$$
where $K_G(I_GY)_c
= \bigoplus 
K_G(I(\Psi))_{c(\Psi)}$. 
The inverse is
${f_*\over{\euler{N_f}}}$.

Applying this construction with $Y = I_GX$ allows us to define a product we call \emph{the  localized orbifold product}
\begin{df}
The \emph{ localized orbifold product} on $K_G(I_GX) \otimes {\mathbb C}$ is defined by the formula 
\begin{equation*}
\alpha \star_{LO} \beta = If_*\left( (If^*\alpha \star If^*\beta) 
\otimes \euler(N_{If})^{-1}\right),
\end{equation*}
where $\star$ is the usual orbifold product on 
$K_G(I_G(I_GX))_c$,
and $If \colon 
I_G(I_GX)
 \rTo I_GX$  is the projection.\end{df}
\begin{rem}
It should be noted that 
$I_G(I_GX)$ 
is not the same as 
$\Itwo_GX$. The inertia $
I_G(I_GX)
 = \{(h,g,x) | hx = gx = x, hg = gh\}$ is a closed
subspace of $\Itwo_GX$.
\end{rem}
The  localized product can be interpreted in the context of the 
$\star_{V^+}$ product,
where the vector bundle $V$ is replaced by the virtual bundle $-N_f$. 
Observe that the pullback of $\T$ to $I_GX$ splits as $\T = 
\T_{\IGX}
 + N_f$,
where $N_f$ is the normal bundle to the finite l.c.i.~map $I_GX \rTo X$.
Although $N_f$ is not a bundle on $X$, we can still compute $N_f^+(g_1,g_2)$
on $\Itwo_G X$. 

The same formal argument used in the proof of Theorem \ref{prop.star+geo} yields the following result.
\begin{prop}
The class 
$\euler(LR(\T) + \ar^+(-N_f))$ 
is well defined in localized $K$-theory
and $$\alpha \star_{LO} \beta = \alpha \star_{(-N_f)^+} \beta.$$
\end{prop}
\begin{rem}
  The inertial pair corresponding the  localized product is the
  formal pair $(LR(\T) + R^+(-N_f),{\cs}\T + {\as}^+(-N_f))$. However,
  the Chern character corresponding to this inertial pair is the usual
  orbifold Chern character and the corresponding product on $A^*(I\ix)$
is the usual orbifold product. The reason is that the orbifold Chern character
isomorphism
  factors through $K_G(I_GX)_{(1)}$,  the localization of  
$K_G(I_GX)$ at the augmentation ideal of $R(G)$. This localization corresponds to the untwisted sector of $I_GX$ where $f$ restricts to the identity map. 
\end{rem}
\begin{rem}
Identifying $K_G(I_GX)_\nc$ with the localization of $K_G(I_G(I_GX))_\nc$ allows us to invert the class $\euler(N_f)$. In \cite[Section 3.4]{Kau:10a} the author gives a framework for defining similar products after formally inverting the Euler classes of normal bundles.
\end{rem}
\subsection{An example with $\Pro(1,2)$}
We consider the weighted projective line $\ix = \Pro(1,2) = 
[(\A^2\smallsetminus \{0\})/{\mathbb C}^*]$, 
where ${\mathbb C}^*$ acts with
weights $(1,2)$.  The inertia stack $I\ix$ has two sectors $\ix^1 =
\ix$ and $\ix^{-1} = \B\mu_2$. We have $K(\ix^1) 
\otimes\nc
= {\mathbb  C}[\chi]/\left<(\chi -1)(\chi^2-1)\right>$ and 
$K(\ix^{-1})\otimes\nc = {\mathbb C}[\chi]/\left<\chi^2-1\right>$.  In particular 
$K(I\ix) \otimes {\mathbb C}$  
is supported at $\pm 1  \in {\mathbb C}^*$. As was the case in
Section \ref{sec.pro13}, inertial ring structures are determined by
the products $\one_{g_1} \star \one_{g_2} \in K(\ix^{g_1g_2})$. In
terms of the localization decomposition, $K(I\ix) \otimes\nc
= K(I\ix)_{(1)}
\oplus K(I \ix)_{(-1)}$. 
The  localized product is determined by computing the corresponding
orbifold product on each localized piece using the decomposition of
the element $\one_{g}$ into its localized pieces and the
product $\one_{g_1} \star_{LO} \one_{g_2}$ decomposes as
$$(\one_{g_1})_{(1)} \star_{LO} (\one_{g_2})_{(1)} + 
(\one_{g_1})_{(-1)} \star_{LO} (\one_{g_2})_{(-1)}.$$

The multiplication matrix for $K(I\ix)_{(1)}$
is the usual orbifold matrix, which in this case is the following.

\begin{center}
\begin{tabular}{l|cc}
 & $\ix^1$ & $\ix^{-1}$\\
\hline
$\ix^1$  & 1 & 1 \\
$\ix^{-1}$   & 1 & $\euler(\chi)$\\
\end{tabular}
\end{center}

The multiplication matrix for the  localized product on $K(I\ix)_{(-1)}$
is the same as the multiplication matrix for the orbifold product on
$\B\mu_2$ which is the following.

\begin{center}
\begin{tabular}{l|cc}
& $\ix^1$ & $\ix^{-1}$\\
\hline
$\ix^1$ &  1 & 1 \\
$\ix^{-1}$ & 1 & 1\\
\end{tabular}
\end{center}

Thus we see that the only nontrivial product is 
$\one_{(-1)} \star_{LO} \one_{(-1)}$.
To obtain a single multiplication matrix we use the decomposition
$$\one_{(-1)} = (1+ \chi)/2 + (1-\chi)/2 \in K(\ix^{-1})\otimes\nc,$$
where $(1+\chi)/2$ is supported at $1$ and $(1-\chi)/2$ is supported
at $-1$. The final result is the following.
\begin{center}
\begin{tabular}{l|cc}
& $\ix^1$ & $\ix^{-1}$\\
\hline
$\ix^1$ & 1 & 1 \\
$\ix^{-1}$ & 1 & ${(1+\chi)^2\euler(\chi) + (1-\chi)^2\over{4}}$
\end{tabular}
\end{center}
Because the twisted sector $\ix^{-1}$ has dimension 0, both the orbifold 
and usual Chern characters on this sector compute the virtual rank. The untwisted
sector $\Pro(1,2)$ has Chow ring ${\mathbb C}[t]/\langle t^2 \rangle$ 
where             $t = c_1(\chi)$.
Thus $\ch(\euler(\chi)) = t \in A^*(\Pro(1,2)) \otimes {\mathbb C}$.
Observe that 
$$\ch\left({(1+\chi)^2\euler(\chi) + (1-\chi)^2\over{4}}\right)
= {(2 + t)^2(t) + (-t)^2\over{4}} = t$$ in ${\mathbb C}[t]/\langle t^2 \rangle$
as well.

\nocite{EdGr:00,ARZ:08,HuWa:11, GHHK:11, BeUr:07, BGNX:11, BGNX:07}

\begin{thebibliography}{DHVW86}

\bibitem[AGV02]{AGV:02}
D.~Abramovich, T.~Graber, and A.~Vistoli.
\newblock Algebraic orbifold quantum products.
\newblock In {\em Orbifolds in mathematics and physics ({M}adison, {WI},
  2001)}, volume 310 of {\em Contemp. Math.}, pages 1--24. Amer. Math. Soc.,
  Providence, RI, 2002, arXiv:math/0112004.

\bibitem[AR03]{AdRu:03}
A.~Adem and Y.~Ruan.
\newblock Twisted orbifold {K}-theory.
\newblock {\em Comm. Math. Phys.}, 237(3):533--556, 2003, arXiv:math/0107168.

\bibitem[ARZ08]{ARZ:08}
A.~{Adem}, Y.~{Ruan}, and B.~{Zhang}.
\newblock {A Stringy Product on Twisted Orbifold K-theory}.
\newblock 2008, arXiv:math/0605534v2.

\bibitem[BGNX07]{BGNX:07}
K.~Behrend, G.~Ginot, B.~Noohi, and P.~Xu.
\newblock String topology for loop stacks.
\newblock {\em C. R. Math. Acad. Sci. Paris}, 344(4):247--252, 2007.

\bibitem[BGNX12]{BGNX:11}
K.~Behrend, G.~Ginot, B.~Noohi, and P.~Xu.
\newblock String topology for stacks.
\newblock {\em Ast\'erisque}, 343, 2012, arXiv:math/0712.3857v2.

\bibitem[BU09]{BeUr:07}
E.~Becerra and B.~Uribe.
\newblock Stringy product on twisted orbifold {$K$}-theory for abelian
  quotients.
\newblock {\em Trans. Amer. Math. Soc.}, 361(11):5781--5803, 2009.

\bibitem[CH06]{ChHu:06}
B.~Chen and S.~Hu.
\newblock A de{R}ham model for {C}hen-{R}uan cohomology ring of abelian
  orbifolds.
\newblock {\em Math. Ann.}, 336(1):51--71, 2006.

\bibitem[CR02]{ChRu:02}
W.~Chen and Y.~Ruan.
\newblock Orbifold {G}romov-{W}itten theory.
\newblock In {\em Orbifolds in mathematics and physics ({M}adison, {WI},
  2001)}, volume 310 of {\em Contemp. Math.}, pages 25--85. Amer. Math. Soc.,
  Providence, RI, 2002.

\bibitem[CR04]{ChRu:04}
W.~Chen and Y.~Ruan.
\newblock A new cohomology theory of orbifold.
\newblock {\em Comm. Math. Phys.}, 248(1):1--31, 2004.

\bibitem[CS99]{ChSu:99}
M.~{Chas} and D.~{Sullivan}.
\newblock {String Topology}.
\newblock {\em ArXiv Mathematics e-prints}, November 1999, arXiv:math/9911159.

\bibitem[DHVW85]{DHVW:85}
L.~Dixon, J.~A. Harvey, C.~Vafa, and E.~Witten.
\newblock Strings on orbifolds.
\newblock {\em Nuclear Phys. B}, 261(4):678--686, 1985.

\bibitem[DHVW86]{DHVW:86}
L.~Dixon, J.~Harvey, C.~Vafa, and E.~Witten.
\newblock Strings on orbifolds. {II}.
\newblock {\em Nuclear Phys. B}, 274(2):285--314, 1986.

\bibitem[EG00]{EdGr:00}
D.~Edidin and W.~Graham.
\newblock {R}iemann-{R}och for equivariant {C}how groups.
\newblock {\em Duke Math. J.}, 102(3):567--594, 2000.

\bibitem[EG05]{EdGr:05}
D.~Edidin and W.~Graham.
\newblock Nonabelian localization in equivariant {$K$}-theory and
  {R}iemann-{R}och for quotients.
\newblock {\em Adv. Math.}, 198(2):547--582, 2005.

\bibitem[EHKV01]{EHKV:01}
D.~Edidin, B.~Hassett, A.~Kresch, and A.~Vistoli.
\newblock Brauer groups and quotient stacks.
\newblock {\em Amer. J. Math.}, 123(4):761--777, 2001.

\bibitem[EJK10]{EJK:10}
D.~Edidin, T.~J. Jarvis, and T.~Kimura.
\newblock Logarithmic trace and orbifold products.
\newblock {\em Duke Math. J.}, 153(3):427--473, 2010, arXiv:math/0904.4648.

\bibitem[EJK12]{EJK:12b}
D.~{Edidin}, T.~J. {Jarvis}, and T.~{Kimura}.
\newblock {Chern Classes and Compatible Power Operations in Inertial K-theory}.
\newblock September 2012, arXiv:math/1209.2064.

\bibitem[FG03]{FaGo:03}
B.~Fantechi and L.~G{\"o}ttsche.
\newblock Orbifold cohomology for global quotients.
\newblock {\em Duke Math. J.}, 117(2):197--227, 2003, arXiv:math/0104207.

\bibitem[GHHK11]{GHHK:11}
R.~Goldin, M.~Harada, T.~S. Holm, and T.~Kimura.
\newblock The full orbifold {$K$}-theory of abelian symplectic quotients.
\newblock {\em J. K-Theory}, 8(2):339--362, 2011.

\bibitem[GLS{\etalchar{+}}07]{GLSUX:07}
A.~Gonz{\'a}lez, E.~Lupercio, C.~Segovia, B.~Uribe, and M.~A. Xicot{\'e}ncatl.
\newblock Chen-{R}uan cohomology of cotangent orbifolds and {C}has-{S}ullivan
  string topology.
\newblock {\em Math. Res. Lett.}, 14(3):491--501, 2007.

\bibitem[HW13]{HuWa:11}
J.~Hu and B.-L. Wang.
\newblock Delocalized {C}hern character for stringy orbifold {K}-theory.
\newblock {\em Trans. Amer. Math. Soc.}, 365(12):6309--6341, 2013.

\bibitem[JKK07]{JKK:07}
T.~J. Jarvis, R.~Kaufmann, and T.~Kimura.
\newblock Stringy {K}-theory and the {C}hern character.
\newblock {\em Invent. Math.}, 168(1):23--81, 2007, arXiv:math/0502280.

\bibitem[Kau02a]{Kau:pers}
R.~M. Kaufmann.
\newblock Personal Communication, 2002.

\bibitem[Kau02b]{Kau:02}
R.~M. Kaufmann.
\newblock Orbifold {F}robenius algebras, cobordisms and monodromies.
\newblock In {\em Orbifolds in mathematics and physics ({M}adison, {WI},
  2001)}, volume 310 of {\em Contemp. Math.}, pages 135--161. Amer. Math. Soc.,
  Providence, RI, 2002.

\bibitem[Kau03]{Kau:03}
R.~M. Kaufmann.
\newblock Orbifolding {F}robenius algebras.
\newblock {\em Internat. J. Math.}, 14(6):573--617, 2003.

\bibitem[Kau04]{Kau:04a}
R.~M. Kaufmann.
\newblock Second quantized {F}robenius algebras.
\newblock {\em Comm. Math. Phys.}, 248(1):33--83, 2004.

\bibitem[Kau10]{Kau:10a}
R.~M. Kaufmann.
\newblock Global stringy orbifold cohomology, {K}-theory and de {R}ham theory.
\newblock {\em Lett. Math. Phys.}, 94(2):165--195, 2010.

\bibitem[KP09]{KaPh:09}
R.~M. Kaufmann and D.~Pham.
\newblock The {D}rinfel\cprime d double and twisting in stringy orbifold
  theory.
\newblock {\em Internat. J. Math.}, 20(5):623--657, 2009.

\bibitem[LUX08]{LUX:08}
E.~Lupercio, B.~Uribe, and M.~A. Xicotencatl.
\newblock Orbifold string topology.
\newblock {\em Geom. Topol.}, 12(4):2203--2247, 2008.

\bibitem[PPTT11]{PPTT:07}
M.~J. Pflaum, H.~B. Posthuma, X.~Tang, and H.-H. Tseng.
\newblock Orbifold cup products and ring structures on {H}ochschild
  cohomologies.
\newblock {\em Commun. Contemp. Math.}, 13(1):123--182, 2011.

\bibitem[Rua06]{Rua:06}
Y.~Ruan.
\newblock The cohomology ring of crepant resolutions of orbifolds.
\newblock In {\em Gromov-{W}itten theory of spin curves and orbifolds}, volume
  403 of {\em Contemp. Math.}, pages 117--126. Amer. Math. Soc., Providence,
  RI, 2006.

\bibitem[TX06]{TuXu:06}
J.-L. Tu and P.~Xu.
\newblock Chern character for twisted {$K$}-theory of orbifolds.
\newblock {\em Adv. Math.}, 207(2):455--483, 2006, arXiv:math/0505267.

\end{thebibliography}

\newcommand{\etalchar}[1]{$^{#1}$}
\def\cprime{$'$}


\end{document}